\documentclass[10pt]{amsart}
\usepackage{amssymb,bbm,bm,graphicx} % last two for the bold small k and  bold Greek
\usepackage[all]{xy}
\newcommand{\mc}{\mathcal}

\newcommand{\bb}{\mathbb}

\newcommand{\N}{\bb N}
\newcommand{\Z}{\bb Z}

\newcommand{\R}{\bb R}

\newcommand{\Nn}{\mathrm{N}}
\newcommand{\Pp}{\mathrm{P}}
\newcommand{\NN}{\mathcal{N}}
\newcommand{\PP}{\mathcal{P}}
\newcommand{\pos}{\mathrm{pos}}
\newcommand{\enc}{\mathrm{enc}}
\newcommand{\gens}{\operatorname{gens}}

\newcommand{\varA}{variant~A}
\newcommand{\varB}{variant~B}
\newcommand{\varC}{variant~C}

\DeclareRobustCommand{\qedify}[1]{%
  \ifmmode \quad\hbox{#1}
  \else
    \leavevmode\unskip\penalty9999 \hbox{}\nobreak\hfill
    \quad\hbox{#1}%
  \fi
}
\newtheorem{theorem}{Theorem}[section]

\newtheorem{lemma}[theorem]{Lemma}

\newtheorem{corollary}[theorem]{Corollary}

\theoremstyle{definition}

\newtheorem{Example}[theorem]{Example}
\newenvironment{example}[1][]{\begin{Example}[#1]\pushQED{\qedify{$\diamondsuit$}}}{\popQED\end{Example}}

\theoremstyle{remark}

\numberwithin{equation}{section}

\begin{document}

\title{Lattice games without rational strategies} 
\author{Alex Fink}
\begin{abstract}
We show that the lattice games of Guo and Miller support universal
computation, disproving their conjecture that all lattice games have
rational strategies.  We also state an explicit counterexample to that conjecture:
a three dimensional
lattice game whose set of winning positions does not have a rational
generating function.
\end{abstract}
\maketitle

\section{Introduction}

% Say grander things here!  One should view this work as a link
% in tying combinatorial game theory into mainstream mathematics.
% Lattice games were a way to do that, but what this shows is
% that the power of the resulting theory is too high; lattice games
% and cellular automata are computationally equivalent.

Combinatorial game theory is concerned with sequential-play games with perfect information.
A class of games of fundamental interest are the {\em impartial games},
the games with an acyclic directed graph of 
{\em positions}, in which two players alternate moving from
the current position to one of its {\em options}, i.e.\ its neighbours along out-edges.  
We will insist that the graph be locally finite in the sense that
there are only finitely many positions reachable by sequences of moves from each given position.
We are interested in ascertaining the {\em outcome} of each position,
that is, who will win with perfect play, the player who is next to move
or the other player.
We call a position an {\em $\Nn$-position} (for {\bf n}ext player win) in the former case,
and a {\em $\Pp$-position} (for {\bf p}revious player win) in the latter.

The first theorem that can be said to belong to combinatorial game theory generally,
as opposed to some particular game,
is the Sprague-Grundy theorem~\cite{Sprague,Grundy}.  It asserts that,
under the {\em normal play} convention in which a player with no available
move loses, every position in an impartial game is equivalent to
a position in one-heap Nim.  
The positions of one-heap Nim are $\N=\{0,1,\ldots\}$, and the options 
of~$n$ are $\{0,\ldots,n-1\}$.
The notion of equivalence invoked here is the following.
The {\em disjunctive sum} $G+H$ of two combinatorial games $G$ and $H$
is the game whose graph of positions is the Cartesian product of these
graphs for $G$ and $H$; that is to say, a position of the sum
consists of a position of $G$ and a position of $H$, and a
move in the sum consists of choosing one of $G$ and $H$ and moving in it.  
Two impartial games $G$ and $G'$ are then {\em equivalent} if $G+H$ and $G'+H$
have the same outcome for every impartial game $H$.

In order for this theory, or any other, 
to be of practical use for a given game, it should provide a 
computationally efficient way to compute outcomes of positions.  One family
for which the Sprague-Grundy theory achieves this are the {\em octal games} \cite{GS}, 
including several games which have historically been of interest.  
In the recent paper~\cite{GM}, Guo and Miller propose the class 
of {\em lattice games} as an extension of the octal games, 
looking to use the theory of lattice points in polyhedra, especially
its algorithmic aspects, to obtain computational efficiency.  

The positions of a lattice game are an ideal $\mc B$
with finite complement in a pointed affine semigroup $\Lambda$.
The elements $\Lambda\setminus\mc B$ are called the {\em defeated positions},
though they are not positions as we use the term.
The other datum of a lattice game is its {\em ruleset}, a set 
$\Gamma$ of vectors in $\Z\Lambda$.  See~\cite{GM} for the precise axioms
that are required of~$\Gamma$.
A position $p$ is an option of a position $q$ if and only if $q-p\in\Gamma$; 
that is, a move in a lattice game consists of
subtracting a vector in $\Gamma$ from a position.
(It has been allowed that $\mc B$ not equal $\Lambda$ in order to subsume also
the {\em mis\`ere play} convention, in which a player with no moves wins rather than losing.
Mis\`ere play has proved much less tractable than the normal play to analyze.
However, weakening the notion of equivalence to restrict $H$ to positions of the
game in question leads to the {\em mis\`ere quotient} theory of 
Plambeck and Siegel~\cite{PS}, which has made advances in many particular cases.)
%Mis\`ere quotients have 
%permitted the solution of many particular mis\`ere games, but 
%general control over their structure is still largely lacking.
%The games we construct in this paper, like that of Example~\ref{ex:xor realised},
%are perhaps of interest from this perspective.
%  --- But I won't mention these.  They're not good examples,
%  exactly 

The structure on which are hung the algorithmic results of \cite{GM} and its 
followup \cite{GM2} is the notion of {\em affine stratification}.  A lattice game
has an affine stratification if its set $\PP$ of $\Pp$-positions 
is a finite disjoint union of modules for affine semigroups of $\Lambda$.  
In the language of discrete geometry,
this is true if and only if $\PP$ is a finite disjoint union of
intersections of rational polyhedra with sublattices of $\Z\Lambda$.
Guo and Miller also invoke the related notion of a {\em rational strategy}:
a game has a rational strategy if its set $\PP$ has rational 
multigraded generating function.  
Having an affine stratification is at least as strong as having a rational strategy, 
by standard facts in polyhedral geometry: see for instance \cite{BW}.  
But in fact the two conditions are equivalent \cite{Speyer-note}.

Guo and Miller conjectured that every lattice game possesses a rational strategy.
The purpose of this note is to provide counterexamples to that conjecture.
\begin{theorem}\label{thm:1}
There exist lattice games, with normal play on the board $\N^3$, with no 
rational strategy. 
\end{theorem}
An explicit example of a game with no rational strategy, 
with rule set of size 28, is constructed in Example~\ref{ex:xor realised}.

In terms of the goals of~\cite{GM}, the import of Theorem~\ref{thm:1}
is that lattice games are too broad a class of games
for efficient computation of strategies to be possible; that is to say,
lattice games themselves retain too much computational power.
It seems possible that the lattice game perspective may only prove 
fruitful in this regard for the {\em squarefree lattice games} of \cite{GM},
which already contain the octal games,
and which are proved in that work to have rational strategies in normal play.  
It is perhaps unsurprising that squarefree lattice games should be
amenable to this approach, in view of the Sprague-Grundy theorem
and the fact that it is exactly in the squarefree lattice games 
that for two positions $p,q\in\mc B$, the position $p+q$ that 
is their sum in the semigroup $\Lambda$ is also their disjunctive sum.
Needless to say, the lattice games we construct are not squarefree.

The construction that underlies Theorem~\ref{thm:1}
works by arranging that the outcomes of a certain subset 
of the board should compute a given recursively-defined function 
$f$ on a rank 2 affine semigroup,
% $f:M\to\Sigma$, where $\Sigma$ is a finite set and $M$ is a module for some subsemigroup of $\N^2$,
i.e.\ a function to a set $\Sigma$ satisfying
\[f(\ell) = g(f(\ell-\beta_1),\ldots,f(\ell-\beta_r))\]
% whenever $\ell-\beta_i\in M$ for all $i$, 
for some $g:\Sigma^r\to \Sigma$.
(See Section~\ref{sec:construction} for a more careful statement.)
% and where the $\beta_i\in L$ are vectors
% such that there exists a linear functional pairing positively with all the $\beta_i$
% and the generators of the recession cone of $\conv(L^+)$.
Among the functions of this form are the functions recording the states of the cells
of 1-dimensional cellular automata; therefore, lattice games can 
realize any phenomenon that 1-dimensional cellular automata can.
In particular, since cellular automata are known to be capable of 
universal computation (see \cite{Smith} and several subsequent works), 
the same is true of lattice games.
\begin{theorem}\label{thm:universal}
The class of lattice games in $\N^3$ is capable of universal computation.  
\end{theorem}

We will describe variants of our construction allowing
the input to the cellular automaton to be provided in either of the following ways:
\begin{enumerate}
\renewcommand{\labelenumi}{(\Alph{enumi})}
\item one ruleset which is computationally universal, with the input to
the universal function encoded in the defeated positions;
\item a family of rulesets which themselves encode the input,
in normal play (i.e.\ with no defeated positions).
\end{enumerate}
In both cases the output of the computation is given by the outcomes of certain positions.
In particular we get the following corollary.
\begin{corollary}\mbox{}
\begin{enumerate}
\renewcommand{\labelenumi}{(\alph{enumi})}
\item There exists a ruleset for a lattice game on $\N^3$ such that, given a set of defeated positions
$\mc D$ and and integers $m, a, b$,
it is undecidable whether there exist $i,j\in\N$ such that $(mi+a,mj+b,1)$ is a $\Pp$-position.
\item Given a ruleset for a lattice game on $\mc B=\N^3$ with no defeated positions, and integers $m, a, b$,
it is undecidable whether there exist $i,j\in\N$ such that $(mi+a,mj+b,1)$ is a $\Pp$-position.
\end{enumerate}
\end{corollary}

\subsection*{Conventions}
For sets $P$ and $Q$, $P+Q$ is the Minkowski sum, $-Q = \{-q:q\in Q\}$, 
and $P-Q=P+(-Q)$.  This last is not the set difference, which is written $P\setminus Q$.

For lattice games, we follow the notation and definitions of~\cite{GM}.
We will supplement this by speaking of the outcome function
$o:\mathcal B\to O$ from the positions $\mathcal B$ of a lattice game 
to the set of two symbols $O=\{\Pp,\Nn\}$; this is defined by
$o(p)=\Pp$ if $p$ is a $\Pp$-position, and $o(p)=\Nn$ if $p$ is an $\Nn$-position.
(In~\cite{GM}, the authors mostly speak of bipartitions; they would write
$p\in\PP$ and $p\in\NN$ to express these.)  
%We also use the conventional names {\em $\Pp$-position} for an element of $\PP$,
%and {\em $\Nn$-position} for an element of $\NN$.

\section{The construction}\label{sec:construction}

In this section, the remainder of the paper, we provide a construction proving 
Theorems \ref{thm:1} and~\ref{thm:universal}, and close with a concrete example of 
a lattice game without a rational strategy.
In fact, we will describe three closely related variants of the construction at once.  
Two of them, variants~A and~B, correspond to the two ways of providing the input of a computation
mentioned just after Theorem~\ref{thm:universal}.  The third, call it variant~C for completeness,
will correspond to using neither of these and not providing for any sort of input.

We first set up the recursively defined function $f$ which we will be arranging
for the outcomes of our game to compute.
Let $L$ be a sublattice of $\Z^2$, $L^+=L\cap\N^2$, and let $M$ be an ideal in $L^+$
with finite complement.
Let $\beta_1,\ldots,\beta_r\in L$ be nonzero vectors in the same halfspace,
which themselves satisfy the tangent cone axiom for~$\N^2$.  
Let $\Sigma$ be a finite set.
Let $g:\Sigma^r\to\Sigma$ be a map of sets, and define 
$f:M\to\Sigma$ as follows.  
If $\ell\in M$ is a minimal $L^+$-module generator of $M$,
we may let $f(\ell)$ take any value $f_0(\ell)$.
(We will use this in \varA\ to provide the input.)
If $\ell$ is not a generator, then we let
\begin{equation}\label{eq:recursion}
f(\ell) = \left\{\begin{array}{ll}
g(f(\ell-\beta_1),\ldots,f(\ell-\beta_r)) & \mbox{$\ell-\beta_i\in M$ for all $i$}\\
\sigma_0 & \mbox{otherwise}
\end{array}\right.
\end{equation}
for some fixed $\sigma_0\in\Sigma$.

In outline, our construction will produce a periodic circuit of nor gates computing $f$,
and realize this on a union of dilates of $M$ inside the slice $\N^2\times\{1\}$ 
of the board $\N^3$.  Some of the moves in the ruleset will correspond to 
the wires in the circuit.  At the heart of this is the fact that,
if we identify $\Pp$ with true and $\Nn$ with false, 
then nor is the function that computes the outcome of a position
from the outcomes of its options.  The rest of the moves in the ruleset
force the outcomes on $\N^2\times\{0\}$ to have the same period as the circuit construction,
and the outcomes on $\N^2\times\{1\}$ minus the circuit to all be $\Nn$
so as not to interfere with the circuit.  The positions with 
third coordinate $\geq2$ are irrelevant to the construction.

Let $O=\{\Pp,\Nn\}$ be the set of outcomes.  
We begin by choosing an encoding $\enc:\Sigma\hookrightarrow O^s$ for some $r$,
such that $\enc(\sigma_0)=(\Nn,\ldots,\Nn)$.  
Let $\tilde g:O^{rs}\to O^s$ be the corresponding encoding
of $g$, i.e.\ a function such that
$\tilde g(\enc(\sigma_1),\ldots,\enc(\sigma_s) = \enc(g(\sigma_1,\ldots,\sigma_s))$
for all $\sigma_1,\ldots,\sigma_s$.  
Any boolean function may be computed by a circuit of nor gates.  Let $G'$ be a
directed acyclic graph giving such a circuit for $\tilde g$, as follows.  
Among the vertices of $G'$ (the {\em gates}) will be 
$rs$ inputs $in_{11},\ldots,in_{rs}$, with no in-edges, 
and $s$ outputs $out_1,\ldots,out_s$, with no out-edges.  We require that
when $o:V(G')\to O$ is the function such that 
\begin{equation}\label{eq:o(G)0}
o(in_{ij}) = \enc(\sigma_i)_j
\end{equation}
and 
\begin{equation}\label{eq:o(G)1}
o(v)=\mathrm{nor}(v_1,\ldots,v_k)
\end{equation}
for any non-input vertex $v$,
then $o(out_j)=\enc(g(\sigma_1,\ldots,\sigma_s))_j$.
Here $\mathrm{nor}:O^k\to O$ takes the value $\Pp$ on the tuple $(\Nn,\ldots,\Nn)$
and N elsewhere.

We may assume, for each $i$, that $G'$ contains a path from $in_{ij}$ to $out_{j'}$
for some $j$ and~$j'$.  If not, $g$ must not depend on its $i$-th parameter, 
and that parameter may as well not be present.  Moreover, we may choose the encoding $\enc$
so that the following holds: $\enc(g(\sigma_1,\ldots,\sigma_s))_1$ 
depends nontrivially on some $\sigma_i$ such that $(\beta_i)_1\leq 0$, and
$\enc(g(\sigma_1,\ldots,\sigma_s))_s$ 
depends nontrivially on some $\sigma_{i'}$ such that $(\beta_{i'})_2\leq 0$.
Then there are corresponding paths in $G'$ to $out_1$ and $out_s$.
(These $\beta$s exist by the tangent cone axiom, and this last assumption will in its turn
only be used to show the tangent cone axiom.)  

Let $G$ be $G'$ together with an additional vertex $in'$ and edges
$in'\to out_j$ for each $j$, as well as, 
for \varB,
another vertex $in''$ with edges to all of the $out_j$ and all $v$ such that $v\to out_j$. 
If $in'$ is treated as an extra input, and $in''$ is ignored for the moment, 
$G$ is a nor circuit for
\begin{equation}\label{eq:G circuit}
(\enc(\sigma_1),\ldots,\enc(\sigma_s),b)\mapsto\left\{\begin{array}{ll}
\enc(g(\sigma_1,\ldots,\sigma_s)) & b=\Nn \\
(\Nn,\ldots,\Nn) = \enc(\sigma_0) & b=\Pp \\
\end{array}\right..
\end{equation}
Eventually, the dependence on $b$ here will reproduce the choice between the cases
in~\eqref{eq:recursion}.

Next, we choose a way $\pos:V(G)\to\Z^2$ to position the gates in $\Z^2$,
as well as a positive integer $m$, 
and a nonempty finite subset $I\subseteq\N^2$, which is the complement of an ideal.
We impose the following conditions on $\pos$.
\begin{enumerate}
\renewcommand{\labelenumi}{(\alph{enumi})}
\item The differences $\pos(w)-\pos(v)$ for $v\to w$ an edge of $G$,
and the set $(\N^2-I)\setminus(I-I)$,
all lie in the same open halfspace $H$.
\item $\pos(in_{ij}) = \pos(out_j) - m\beta_i$ for all $i,j$.
\item If $w\in V(G)$ is not one of the $in_{ij}$, then for all $v\in V(G)$ 
such that $\pos(w)-\pos(v)\equiv \pos(w')-\pos(v')$ modulo $mL$
for some edge $w'\to v'$ of~$G$, there must exist a vertex $v''$
such that $\pos(w)-\pos(v'')=\pos(w')-\pos(v')$; 
and $v''\to w$ is an edge of $G$ for every such $v''$.
% For each edge $v\to w$ of~$G$, if $v',w'\in V(G)$ are such that
% $\pos(w)-\pos(v)\equiv \pos(w')-\pos(v')$ modulo $mL$
% and $w'$ is not one of the $in_{ij}$, 
% then $\pos(w')-\pos(v')=\pos(w)-\pos(v)$, and $v'\to w'$ is an edge of~$G$.  
\item The set
$(\{\pos(w)-\pos(v):v,w\in V(G)\} \cup (I-I))+mL$
contains no translate $p-I$ of~$-I$ other than the translates with
$p\in I+mL$ (which it manifestly does).
\item The set $\pos(V(G))+mL$ contains no translate of a set of form
$\{p-h(p):p\in I\}$ for some function $h:I\to I$ with $h(p)\neq p$ for all $p$.
\item For $v\to w$ an edge of $G$, $\pos(w)-\pos(v)$ does not lie in $(I-I)+mL$.  
\item No other $\pos(v)$ lies in $\pos(in')+I+mL$ or in $\pos(in'')+I+mL$.
The differences $\pos(w)-\pos(in'')$ for $in''\to w$ an edge are congruent
to no other differences $\pos(w')-\pos(v')$ modulo $mL$.  
\item $\pos(in')\in\N^2$ and $\pos(in'')\in\N^2$, but $\pos(in_{ij})\not\in\N^2$ for any $i,j$.
\item $\pos(out_j)_1<\pos(out_{j'})_1$ and $\pos(out_j)_2>\pos(out_{j'})_2$ for $j<j'$.
If $v\to out_j$ is an edge of $G$, then $v\not\in\pos(out_{j'})+\N^2$ for any $j'$.
\end{enumerate}

These have roughly the following functions: (a) ensures pointedness of the constructed ruleset;
(b) ensures that the circuit can be overlaid with itself as necessary to implement the recurrence;
(c) through~(f) handle aspects of obtaining the requisite outcomes of the positions
that are not gates; and (g) through~(i) ensure that the initial conditions come out as 
they should, with invocations of $in''$ only relevant in \varB\ and (i) only in \varA.

Condition~(e) implies that the complement $\Z^2\setminus(\pos(V(G))+mL)$
is a union of translates of~$I$.  For if it weren't possible
to cover $q\in \Z^2\setminus(\pos(V(G))+mL)$ with a translate of~$I$,
the elements of $\pos(V(G))+mL$ meeting each of the possible translates
of $I$ containing $q$ would form a translate of some $\{p-h(p):p\in I\}$.

Condition~(h) implies that $\pos(out_i)\in\N^2$ for all $i$.
Condition~(a) implies that $(1,0)$ and $(0,1)$ lie in the given halfspace.

\begin{lemma}\label{lem:possible}
There exists some $(\pos,m,I)$ satisfying these conditions.
\end{lemma}

\begin{proof}
Fix an open halfspace $H$ in which all the $\beta_i$ lie; this will be the halfspace
we use for condition~(a).
Fix any suitable $I$ of cardinality $>1$ which satisfies the restriction on~$I$
stated there; for instance, take $H$ to have a normal vector $(x,y)\in\N^2$,
and let $I$ be $\{(i,j):0\leq i\leq y,0\leq j\leq x\}$.
The remaining choices can be thought of as choosing an integer point 
$(\pos,m)$ from the space $\R^{2V(G)+1}$.

Condition~(b) means that $(\pos,m)$ must lie in a certain linear subspace $V$.
Conditions (a) and~(h) and~(i) impose linear inequalities.
We will supplement these with the further linear inequalities that each coordinate of each $\pos(v)$
is bounded in absolute value by $2mB$ where $B$ is the maximal absolute value of a coordinate of $\beta_i$.
These cut out a cone $C\subseteq V$ with nonempty relative interior.
Indeed, order the vertices $v_1,\ldots,v_{|V(G)|}$ of $G$ so that $i<j$ for
each edge $v_i\to v_j$.  Let $\epsilon$ be a small vector in $\operatorname{int}(H\cap\R_{\geq0}^2)$,
and $\zeta$ a much smaller vector with $\zeta_1>0>\zeta_2$.
Then sufficiently small perturbations of the point $m=1$, $\pos(out_j)=|V|\epsilon+j\zeta$, 
$\pos(in_{ij})=\pos(out_j)-\beta_i$, and $\pos(v_i)=i\epsilon$
for $v_i$ not of one of the preceding forms, lie in the interior of this cone.  
Finally, conditions (d), (e), (f), and~(g) can be recast to say that 
$(\pos, m)$ must not lie on some finite collection of proper linear subspaces of~$V$.
To get the finiteness, note that it suffices to state the
conditions with a finite subset of $mL$ in place of $mL$, on account of
our additional bounds on the coordinates of the $\pos(v)$.
The same holds for a strengthening of condition~(c), in which we simply demand 
that $\pos(w)-\pos(v)$ is never congruent to $\pos(w')-\pos(v')$ modulo $mL$
for $(v,w)\neq(v',w')$, except for the unavoidable such congruences 
that are consequences of $in_{ij}\equiv out_j$.

We conclude there must be a lattice point $(\pos,m)$ in $C$ 
lying off all of these finitely many subspaces (for instance, by asymptotics
of Ehrhart functions), and therefore that it's possible to fulfill all the conditions.
\end{proof}

In \varA, we need to specify a set $\mathcal D$ 
of defeated positions.  (Otherwise, there are no defeated positions, that is $\mathcal D=\emptyset$.)  
Recall that $M$ is the $L^+$-module that is the domain of the recursive function $f$.  
Let $\gens M$ be its minimal set of generators as an $L^+$-module;
these are the $\ell$ such that we have allowed $f(\ell)$ to take arbitrary values $f_0(\ell)$.
The set $\mathcal D$ will be a subset of $\N^2\times\{0,1\}$.
Let the $p_3=0$ slice (i.e.\ $\{\ell:(\ell,0)\in \mathcal D\}$) 
be $\N^2\setminus((M\setminus\gens M)+\N^2)$,
and the $p_3=1$ slice be the $p_3=0$ slice minus the union of 
$\N^2+(\pos(out_j)+m\ell)$ for every $\ell\in\gens M$ and $j\in[s]$
such that $\enc(f_0(\ell))_j = \Pp$.

%\enlargethispage{2\baselineskip} % kluge!

Let $F$ be the complement of the ideal in $\N^2$ generated by 
nonzero elements of $mL^+$; then $F$ contains a representative of
every class in $\Z^2/mL$.  
Let $B'$ be a set of $L^+$-module generators for $\bigcap_{i=1}^r (\beta_i+L^+)$,
and $B''$ a set of generators for $L^+\setminus 0$.
Define the ruleset $\Gamma\subseteq\Z^2\times\Z$ by
\begin{align}
\Gamma :=\!\!&\mathbin{\phantom\cup} \{(\pos(w)-\pos(v),0) : v\to w\in E(G), v\neq in'' \}
\label{eq:Gamma wires}
\\&\cup \{(p,0) : p\in F-I, p\not\in (I+mL)-I, 
\label{eq:Gamma slice 0}
\\&\qquad\quad  p\not\equiv \pos(w)-\pos(v)\mbox{ (mod $mL$) for any $v,w\in V(G)$}\}
\notag
\\&\cup \{(p,1) : p\in -m\{\beta_1,\ldots,\beta_r\}+F-I, p+I\subseteq \Z^2\setminus(\pos(V(G))+mL) \}
\label{eq:Gamma slice 1}
\\&\cup \{(\pos(in')+m\ell,1) : \ell\in B'\}
\label{eq:Gamma in'}
\\&\cup \{(\pos(in'')+m\ell,1) : \ell\in B''\}
\label{eq:Gamma in''}
\\&\cup \{(0,0,2)\}
\label{eq:Gamma tangent}
\\&\cup \Gamma_{\mathrm B}.
\label{eq:Gamma''}
\end{align}
If we are not in \varB, then the last term $\Gamma_{\mathrm B}$ is
empty.  If we are, then let $\gens M$ be the unique minimal set of $L^+$-module generators of~$M$,
and set
\begin{multline*}
\Gamma_{\mathrm B}=\{(\pos(v)-\pos(in'')+m\ell,0) : \ell\in \gens M, v\to out_j\in E(G)\} \\
\cup\{(\pos(out_j)-\pos(in'')+m\ell,0) : \ell\in \gens M, \enc(f_0(\ell))_j=\Nn\}.
\end{multline*}

Again, the general functions of these are as follows:
\eqref{eq:Gamma wires} handles the wires in the circuit; \eqref{eq:Gamma slice 0}
makes $\N^2\times\{0\}$ have the correct outcomes, and \eqref{eq:Gamma slice 1} does
similarly for nongates in $\N^2\times\{1\}$; \eqref{eq:Gamma in'} and \eqref{eq:Gamma in''}
make $in'$ and $in''$ behave, respectively; \eqref{eq:Gamma tangent} is only there to
handle a case of the tangent cone axiom; and \eqref{eq:Gamma''} sets up the initial conditions
in~\varB.

With the construction set up, the proofs of Theorems \ref{thm:1} 
and~\ref{thm:universal} reduce to showing the following two lemmas.
\begin{lemma}\label{lem:axioms}
The axioms for a ruleset are satisfied by $\Gamma$.  
\end{lemma}

\begin{lemma}\label{lem:cstr}
For any $\ell\in M$, we have that
\[ \enc(f(\ell)) = (o(\pos(out_1)+m\ell,1),\ldots,o(\pos(out_s)+m\ell,1)) \]
in the lattice game on $\N^3$ with ruleset $\Gamma$ and defeated positions $\mathcal D$.
\end{lemma}

\begin{proof}[Proof of Lemma~\ref{lem:axioms}]
The halfspace $H$ of condition~(a) on the construction
provides a functional on which $\Gamma\cap (\N^2\times\{0\})$, and the generators 
$(1,0,0)$ and $(0,1,0)$ of the board, are both positive.  The finitely many remaining elements of
$\Gamma$, and the last generator $(0,0,1)$ of the board, have third coordinate positive,
so the functional can be lifted from $\N^2$ to $\N^3$ in some way making these all positive.

The tangent cone axiom for the rays $\R_+(1,0,0)$ and $\R_+(0,1,0)$ follows from the 
tangent cone axiom for $\beta$ and the existence of the paths in $G'$ to 
$out_1$ and $out_s$ that we imposed.  To make the argument for $out_1$, 
there is a path in $G'$ from some $in_{ij}$ where $(\beta_i)_1\leq 0$.
This path yields a sequence of moves in $\Gamma$ of form $\pos(w)-\pos(v)$,
whose sum $\pos(out_1)-\pos(in_{ij})$ has nonpositive first coordinate by 
conditions (b) and~(i).  Therefore, one of the moves in the sum must have
nonnegative first coordinate also (and its third coordinate is zero).  
The tangent cone axiom for the ray $\R_+(0,0,1)$ is provided for by the element $(0,0,2)\in\Gamma$.
\end{proof}

For Lemma \ref{lem:cstr}, we will introduce a subsidiary lemma.

\begin{lemma}\label{lem:slice 0}
The $\Pp$-positions on $\N^2\times\{0\}$ are exactly the
non-defeated positions among $((I+mL)\cap\N^2)\times\{0\}$.
\end{lemma}
As a corollary of the definition of the defeated positions,
for $\ell\in L$, it holds that $(m\ell,0)$ is a $\Pp$-position if and only if $\ell\in M$.

\begin{proof}
The elements of $\Gamma$ that have relevance to the outcomes on  $\N^2\times\{0\}$
are those in \eqref{eq:Gamma wires} and \eqref{eq:Gamma slice 0} and~\eqref{eq:Gamma''}.
We proceed by induction on the pairing of a position $p$
with the normal of $H$.  There are two cases.
Let $p\in\N^2$ not be in $I+mL$;
we want to show $(p,0)$ an $\Nn$-position.  
Let $\ell$ be a point of $mL$ in $p-\N^2$ such that $(p-\N^2)\cap(\ell+\N^2)$ contains
no other point of $mL$; there is such an $\ell$, since all $\N^2$-generators of
the set of defeated positions lie in $mL$.  
Then we have $p-\ell\in F$.  Now, $(p-\ell)-I$ must intersect the
set \eqref{eq:Gamma slice 0}; this could only fail to be the case if
every element of $(p-\ell)-I$ had the form $\pos(w)-\pos(v)$ modulo $mL$.   
But that is ruled out by condition~(d).  So there is a move from $(p,0)$ to an element of
$\ell+I$, which is a $\Pp$-position.

For the other case, let $p\in ((I+mL)\cap\N^2)$.  We want to show $(p,0)$ is an $\Nn$-position,
i.e.\ it has no move to a $\Pp$-position, which is by the inductive hypothesis necessarily
of form $(p',0)$ for $p'\in I+mL$.
But no element of $\Gamma$ lies in $(I-I)+mL$.  For \eqref{eq:Gamma wires} and~\eqref{eq:Gamma''}
this is by condition~(f), and for \eqref{eq:Gamma slice 0} it is direct.  
This proves the lemma.
\end{proof}

\begin{proof}[Proof of Lemma~\ref{lem:cstr}]
Suppose $(p,1)\in\N^2$ is not in $\pos(V(G))+mL$
and is not defeated.  It then follows from Lemma~\ref{lem:slice 0} that $(p,1)$ is an $\Nn$-position.
By construction there exists $i$ so that 
$(p+m\beta_i,0)$ is not defeated, and therefore as before there is $\ell\in mL^+$
not defeated such that $p+m\beta_i-\ell \in F$.  Then by condition~(e), one of the
elements of \eqref{eq:Gamma slice 1} is a move from $p$ to an element of $\ell+I$, 
which is a $\Pp$-position.

On the other hand, for $w\in G$ and $\ell\in L^+$, let $p=\pos(w)+m\ell$. 
If $w\in G'$, then there are no moves from $(p,1)$ to a $\Pp$-position
in $\N^2\times\{0\}$: none of \eqref{eq:Gamma slice 1}, \eqref{eq:Gamma in'},
or \eqref{eq:Gamma in''} can provide them, the last two by condition~(g).
This implies that, if $(p,1)$ is an $\Nn$-position with $w\in G'$, it must have a
$\Pp$-position option which is also in $(\pos(V(G))+mL^+)\times\{1\}$.

If $v=in'$ or $v=in''$, then there may be a $\Pp$-position option in $\N^2\times\{0\}$,
by \eqref{eq:Gamma in'} in the former case,
or \eqref{eq:Gamma in''} in the latter.
Indeed, we can read off the definition of these parts of the ruleset,
using our claim,
exactly when these are available.  For $w=in'$, this is when $\ell$ is in 
$M+\bigcap_i(\beta_i+L^+)$, i.e.\ exactly when $f(\ell)$ is defined
by the recursion \eqref{eq:recursion}.
As for $w=in''$, recall that we have only introduced $in''$ into $G$ 
for \varB; in particular, we are
not in \varA, and there are no defeated positions.
Thus in this case $(p,1)$ has a $\Pp$-position option in $\N^2\times\{0\}$
exactly when $\ell\neq 0$.

Continue to let $p=\pos(w)+m\ell$.  The options of $(p,1)$ 
in $(\pos(V(G))+mL)\times\{1\}$ are obtained by subtracting elements of 
\eqref{eq:Gamma wires} and~\eqref{eq:Gamma''}
(we can exclude \eqref{eq:Gamma slice 0}, by definition).
If $w$ is not one of the $in_{ij}$, then by condition~(c),
if $q\in \pos(V(G))+mL$ and $(q,1)$ is an option of~$(p,1)$ via 
subtracting a vector in \eqref{eq:Gamma wires}, 
we can write $q=\pos(v)+m\ell$ for the same $\ell\in L$,
and $v\to w$ an edge of~$G$.  
If $(q,1)$ is an option of $(p,1)$ via subtracting a vector in \eqref{eq:Gamma''},
then $q=\pos(in'')+m\ell$ for some $\ell\in L$, by condition~(g).

Since there are no edges in $G$ to $in'$ or $in''$, the same argument shows
there are no moves to $\Pp$-positions in $\N^2\times\{0\}$ for these vertices.  Hence
\begin{quote}
$(\pos(in')+m\ell,1)$ is $\Pp$ if and only if $\ell-\beta_i\not\in M$ for some $i$; \\
$(\pos(in'')+m\ell,1)$ is $\Pp$ if and only if $\ell= 0$.
\end{quote}

\medskip

To prove the statement of Lemma~\ref{lem:cstr} we use induction,
for $\ell\in M$, on the pairing with the normal of $H$.  

The base case is when $\ell$ is a generator of~$M$
and we have allowed arbitrary initial conditions, $f(\ell)=f_0(\ell)$.
If we are in \varA, then by construction of the defeated positions, 
$(\pos(out_j)+m\ell,1)$ is defeated if and only if $\enc(f_0(\ell))_j=\Nn$.
By condition~(i), if $\enc(f_0(\ell))_j=\Pp$, 
then $(\pos(out_j)+m\ell,1)$ has no undefeated options in $\N^2\times\{1\}$, 
so $(\pos(out_j)+m\ell,1)$ is a $\Pp$-position.

In \varB, there is a vector in \ref{eq:Gamma''} providing $(\pos(in''),1)$ as an 
option of $(\pos(out_j)+m\ell,1)$ exactly when $\enc(f_0(\ell))_j=\Nn$,
as well as for $(\pos(v)+m\ell,1)$ for all generators $\ell$ of $M$ and edges $v\to out_j$.
Since $(\pos(in''),1)$ is a $\Pp$-position, this makes 
$(\pos(out_j)+m\ell,1)$ an $\Nn$-position when $\enc(f_0(\ell))_j=\Nn$,
whereas if $\enc(f_0(\ell))_j=\Nn$ it makes all remaining options of 
$(\pos(out_j)+m\ell,1)$ into $\Nn$-positions, so that $(\pos(out_j)+m\ell,1)$ is a $\Pp$-position.

We also note that in \varB, as since $(\pos(in'')+m\ell,1)$ is an $\Nn$-position
for $\ell\in L^+\setminus\{0\}$, the vectors in the ruleset from \eqref{eq:Gamma''}
play no further role in determining the outcomes.  
So, having handled the generators of $M$, we can proceed as though
$in''$ is not present from now on. 

Now, let $\ell$ not be a generator of $M$.  
For each vertex $w$ of~$G$ other than an $in_{ij}$ or $in'$, 
we have shown that the options of $(\pos(w)+m\ell,1)$ in $\NN\times\{1\}$
are exactly the positions $(\pos(v)+m\ell,1)$ where $v\to w$ is an edge of~$G$.
Thus, if $o:G(V)\to O$ 
is defined such that the outcome of $(\pos(w)+m\ell,1)$ is $o(w)$ for all~$w$,
then $o$ satisfies \eqref{eq:o(G)1},

By the inductive hypothesis, the outcome of 
$\pos(in_{ij}+m\ell,1) = \pos(out_j+m(\ell-\beta_i),1)$ is
$\enc(f(\ell-\beta_i))_j$, if $\ell-\beta_i\in M$.  As noted above,
the outcome of $(\pos(in')+m\ell,1)$ is $\Pp$ if and only if some 
$\ell-\beta_i\not\in M$.
Thus $o:G(V)\to O$ also satisfies \eqref{eq:o(G)0}, and
$o(in')=\Pp$ iff some $\ell-\beta_i\not\in M$.  Accordingly, 
$o$ must compute the function in \eqref{eq:G circuit}.
Comparing to \eqref{eq:recursion}, we have completed the proof.
\end{proof}

\begin{example}\label{ex:xor realised}
We present a lattice game without a rational strategy.
This example deviates from the construction described above in some particulars,
though the fundamental operation of the construction is the same.
This is in the interest of minimality: our last ruleset $\Gamma'$ is the result
of some earnest effort to minimize its size. 

Of note among the deviations is that,
instead of using a vertex $in'$ and the vectors in \eqref{eq:Gamma in'} to detect
the initial conditions of the recurrence, we choose the positions of the gates 
so that $out_1$ itself plays the analogous role.  
In addition, we use \varC, since we do not use the parts of the construction
meant to allow for arbitrary input to computations.

Let $L=\Z^2$, $M=L^+=\N^2$, and let $f:M\to O$ be defined by 
\[f(i,j) = \left\{\begin{array}{ll}
f(i,j-1)\mathbin{\mathrm{xor}} f(i-1,j) & i,j\geq 1\\
\Pp & \mbox{otherwise}
\end{array}\right.\]
Recall our identification $(\Pp,\Nn) = (\mbox{true}, \mbox{false})$.
Thus $\beta_1=(1,0),\beta_2=(0,1)$ in the notation above.

It is easy to check that $f(i,j) = \Pp$ exactly if $\binom{i+j}i$ is odd,
and in particular therefore that, for any natural $n$, we have
$f(i,j)=\Pp$ if $i=2^n>j$ or $j=2^n>i$, but
$f(i,j)=\Nn$ if $i,j<2^n$ while $i+j\geq 2^n$.

Therefore, this game has no affine stratification.  If it had one, 
this would induce an affine stratification on the subset
$\N^2\times\{1\}$.  In particular there
would be some 2-dimensional cone $C\subseteq\R_{\geq0}^2$ such that
the outcome function $C\cap N^2\to O$, $p\mapsto o(p,1)$ factored
through $(C\cap N^2)/L$ for some rank 2 lattice $L$.  
$C$ meets the interior of either $C_1=\R_{\geq0}(1,0)+\R_{\geq0}(1,1)$
or $C_2=\R_{\geq0}(0,1)+\R_{\geq0}(1,1)$, without loss of generality $C_1$;
by intersecting it with $C_1$, we can assume $C\subseteq C_1$.
Let $\ell\in L$ with $\ell_2>0$.  Then, for large enough $n$, there
must exist a point $p=(i,2^n)\in C$ such that $p-\ell\in C$, 
and $(p-\ell)_1+(p-\ell)_2\geq 2^n$.  Then $(p,1)$ is a $\Pp$-position
and $(p-\ell,1)$ is an $\Nn$-position, contradiction.

\medskip

\begin{figure}[ht]
\includegraphics[width=2.5in]{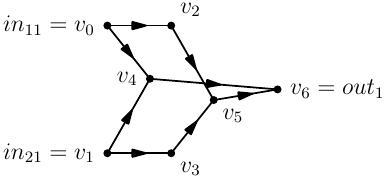}
\caption{A nor circuit computing xor.}\label{fig:xor}
\end{figure}

We take the encoding $\enc:O\to O$ to be the identity.
A nor circuit computing xor is given in Figure~\ref{fig:xor}.
We take $m=6$, $I=\{(0,0),(1,0),(2,0),(0,1)\}$, and 
$\pos$ to be given by the following table.
\begin{center}
\begin{tabular}{r|ccccccc}
     $v $ &   $v_0$ &   $v_1$ &   $v_2$ &   $v_3$ &   $v_4$ &   $v_5$ &   $v_6$ \\
$\pos(v)$ &$(-6, 0)$&$( 0,-6)$&$(-5, 1)$&$( 1,-5)$&$(-1,-2)$&$(-2,-1)$&$( 0, 0)$
\end{tabular}
\end{center}
It is a routine check that this satisfies conditions (a) through~(f) and (i).  
Observe that $\pos(v_2)\equiv\pos(v_3)$ mod $mL$,
so that positions of $\N^2\times\{1\}$ corresponding to one of $v_2$ and $v_3$
will generally correspond to both in different translates; this is not problematic.

We must argue that our construction makes $o(m\ell+\pos(out_1),1)=o(6\ell,1)$
equal to $f(\ell)=\Pp$ when $\ell_1=0$ or $\ell_2=0$, since our proof relied on 
the vertex $in'$ for that.  Our proof still implies that $o(m\ell+\pos(out_1),1)=\Pp$
exactly when there is no edge $v\to out_1$ such that $o(m\ell+\pos(v),1)=\Pp$.
But, if $\ell_1=0$ or $\ell_2=0$, then $m\ell+\pos(v)$ is not in $\N^2$
for any $v$.  Therefore $o(m\ell+\pos(out_1),1)=\Pp$ as desired.  

The ruleset $\Gamma$ turns out to be
\begin{align*}
\Gamma =\!\!&\mathbin{\phantom\cup} 
    \{(1,1,0),(5,-2,0),(-1,4,0),(3,-2,0),(-3,4,0),(1,2,0),(2,1,0)\}
\\&\cup \{(-2,2,0),(-2,4,0),(-1,2,0),(-1,3,0),(0,2,0)(0,3,0)(0,4,0),
\\&\qquad (1,3,0),(1,4,0),(2,2,0),(2,4,0),(3,-1,0),(3,0,0),(3,1,0),
\\&\qquad (3,3,0),(3,5,0),(4,2,0),(4,4,0),(5,2,0),(5,3,0)\}
\\&\cup (\{(-6,0,0),(0,-6,0)\}+\{(-2,1,1),(-2,2,1),(-2,3,1),(-1,2,1),
\\&\qquad (-1,5,1),(0,2,1),(0,3,1),(0,4,1),(0,5,1),(1,-1,1),(1,2,1),
\\&\qquad (1,3,1),(1,4,1),(1,5,1),(2,0,1),(2,1,1),(2,2,1),(2,3,1),
\\&\qquad (2,4,1),(3,0,1),(3,1,1),(3,2,1),(3,3,1),(4,0,1),(4,1,1),
\\&\qquad (4,2,1),(4,3,1),(5,-1,1),(5,0,1),(5,1,1),(5,2,1),(5,5,1)\})
\end{align*}
The three sets here are the lines \eqref{eq:Gamma wires},
\eqref{eq:Gamma slice 0}, and \eqref{eq:Gamma slice 1}.
We leave out the other lines, as these correspond to features of the construction 
we have no need for.

The following smaller ruleset $\Gamma'$ also suffices.  We omit some omissible elements in 
the second and third sets, and as well translate some of the elements of the
third set by $mL$.
\begin{align*}
\Gamma' =\!\!&\mathbin{\phantom\cup} 
    \{(1,1,0),(5,-2,0),(-1,4,0),(3,-2,0),(-3,4,0),(1,2,0),(2,1,0)\}
\\&\cup \{(0,2,0),(0,4,0),(1,4,0),(2,2,0),(2,4,0),(3,0,0),(3,1,0),
\\&\qquad (3,3,0),(3,5,0),(4,2,0),(4,4,0)\}
\\&\cup \{(-2,1,1),(-1,-1,1),(-1,2,1),(0,3,1),(1,-1,1),(1,4,1),(2,0,1),
\\&\qquad (2,1,1),(3,2,1),(4,3,1)\}
\end{align*}

\begin{figure}[ht]
\includegraphics{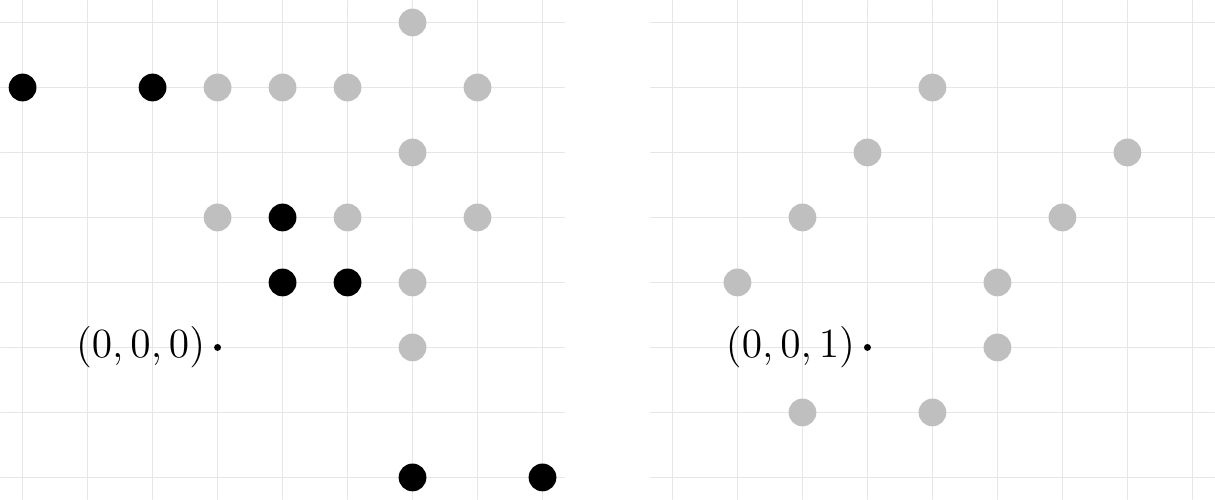}
\caption{The ruleset $\Gamma'$, drawn in the two slices $\Z^2\times\{0\}$,
$\Z^2\times\{1\}$.  Elements on the first line, those corresponding to
wires in the circuit, are black; the others are grey.
(Smaller dots next to labels are not in the ruleset.)}
\label{fig:ruleset}
\end{figure}

This last ruleset $\Gamma'$, of size 28, is portrayed in Figure~\ref{fig:ruleset}.  
The reader or their computer may like to check that the lattice game defined by $\Gamma'$
does indeed have the claimed set of $\Pp$-positions.  Plotting the outcomes of positions $(6i,6j,1)$
will yield the familiar Sierpinski gasket described by $f$, as in Figure~\ref{fig:sierpinski}.

\begin{figure}[ht]
\includegraphics[width=0.475\textwidth]{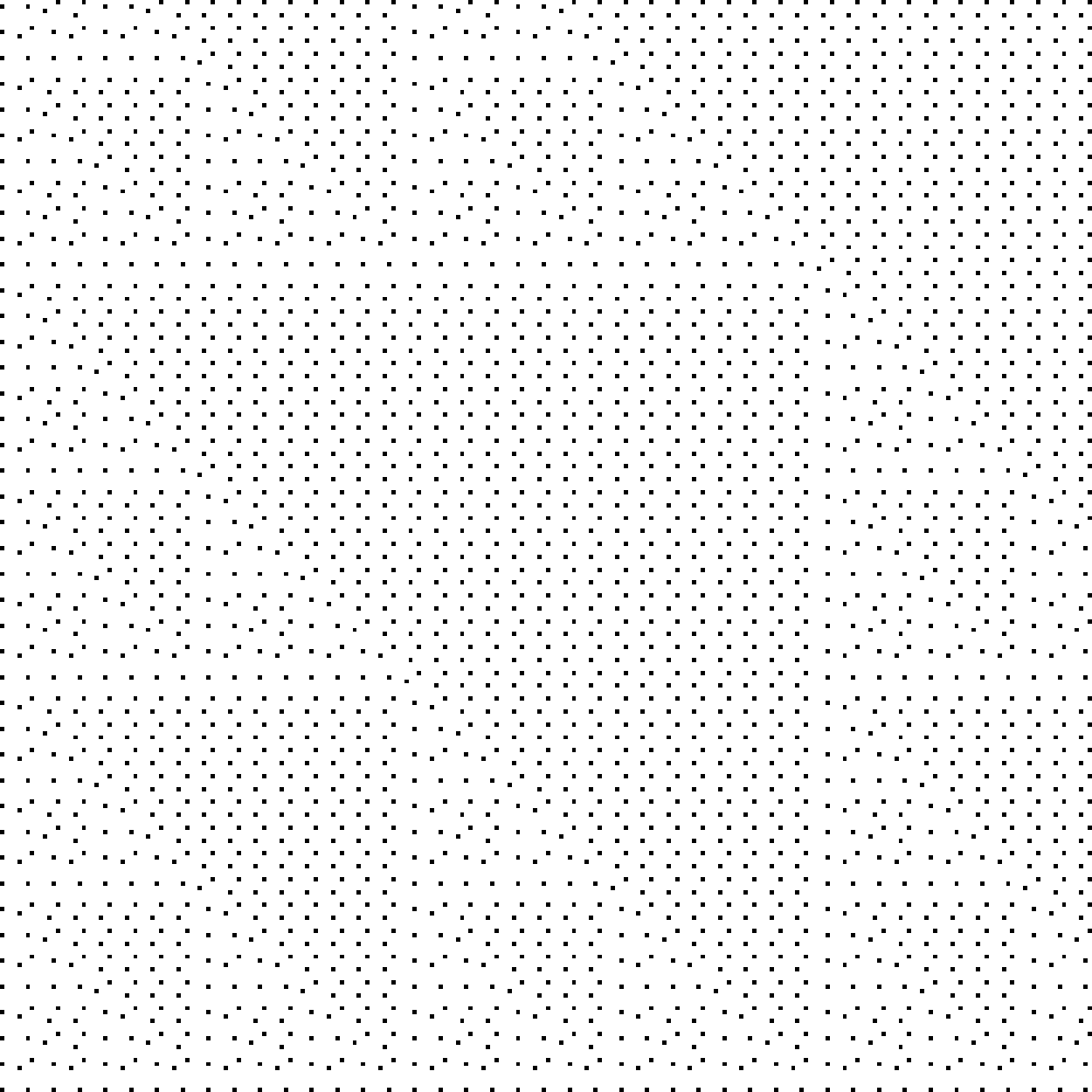}\quad
\includegraphics[width=0.475\textwidth]{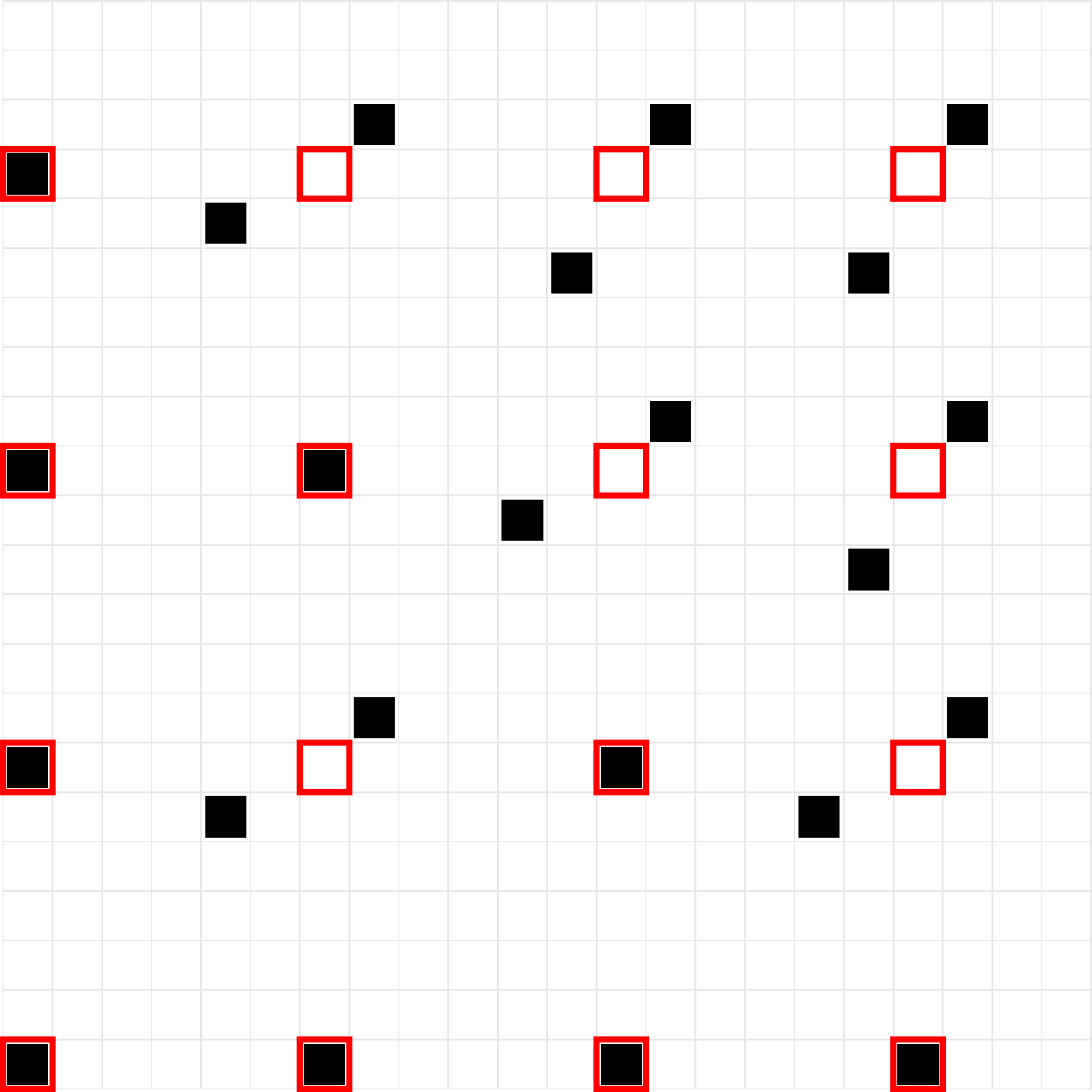}
\caption{Left: the first few outcomes of positions $(x,y,1)$, plotted
in the usual orientation, $\Pp$-positions colored black. Right: a close-up
of the left near the origin, with the positions $(6x,6y,1)$ highlighted.}
\label{fig:sierpinski}
\end{figure}

\end{example}


\begin{thebibliography}{99}
\bibitem{BW} A.~Barvinok and K.~Woods, {\em Short rational generating functions
for lattice point problems}, J.\ Amer.\ Math.\ Soc.\ {\bf 16} (2003), 957--979 (electronic).
\bibitem{Grundy} P.~M.~Grundy, {\em Mathematics and games}, 
Eureka {\bf 2} (1939), 6--8; reprinted in Eureka {\bf 27} (1964), 9--11.
\bibitem{GM} A.~Guo and E.~Miller, {\em Lattice point methods in combinatorial games}, 
Adv.\ in Appl.\ Math.\ {\bf 46} (2011), 363--378.
\bibitem{GM2} A.~Guo and E.~Miller, {\em Algorithms for lattice games}, 
preprint (2011), arXiv:1105.5413v1.
\bibitem{GS} R.~K.~Guy and C.~A.~B.~Smith, {\em The $G$-values of various games}, 
Proc.\ Cambridge Philos.\ Soc.\ {\bf 52} (1956), 514--526.
\bibitem{PS} T.~Plambeck and A.~Siegel, {\em Mis\`ere quotients for impartial games}, 
J.\ Combin.\ Theory Ser.\ A {\bf 115} no.~4 (2008), 593--622. arXiv:math.CO/0609825v5.
\bibitem{Smith} A.~R.~Smith~III, {\em Simple computation-universal cellular spaces}, 
J.\ ACM {\bf 18} no.~3 (1971), 339--353.
\bibitem{Speyer-note} D.~Speyer, {\em Payne's theorem and coefficients of rational
power series}, unpublished, \verb|http://www.math.lsa.umich.edu/~speyer/PowSerNote.pdf|.
\bibitem{Sprague} R.~P.~Sprague, {\em \"Uber mathematische Kampfspiele}, 
T\=ohoku Math.\ J.\ {\bf 41} (1935--1936), 438--444.

\end{thebibliography}
\end{document}